\documentclass[11pt]{amsart}

\usepackage{amsmath}
\usepackage{amssymb}

\newcommand{\dst}{\displaystyle}

\newtheorem{propo}[]{Proposition}
\newtheorem{thm}[]{Theorem}
\newtheorem{lem}[]{Lemma}
\newtheorem{cor}[]{Corollary}

\theoremstyle{remark}

\newtheorem{defi}[thm]{Definition}

\begin{document}

\title[A de Montessus-type theorem]
{Incomplete Pad\'e approximantion and convergence of row sequences of
Hermite-Pad\'e approximants}

\author[J. Cacoq]{J. Cacoq}
\address{Dpto. de Matem\'aticas\\
Escuela Polit\'ecnica Superior \\
Universidad Carlos III de Madrid \\
Universidad 30, 28911 Legan\'es, Spain} \email{jcacoq@math.uc3m.es}
\thanks{The work of  B. de la Calle
received support from MINCINN under grant
MTM2009-14668-C02-02 and from UPM through Research Group ``Constructive Approximation Theory and Applications". The work of  J. Cacoq and G. L\'opez was supported by MINCINN
under grant MTM2009-14668-C01-02}

\author[B. de la Calle]{\hspace*{0.1 cm} B. de la Calle
Ysern}\address{Dpto. de Matem\'atica Aplicada\\
E. T. S.  de Ingenieros Industriales\\
Universidad Polit\'ecnica de Madrid\\
Jos\'e G. Abascal 2, 28006 Madrid, Spain}
\email{bcalle@etsii.upm.es}

\author[G. L\'opez ]{G. L\'opez Lagomasino}
\address{Dpto. de Matem\'aticas\\
Escuela Polit\'ecnica Superior \\
Universidad Carlos III de Madrid \\
Universidad 30, 28911 Legan\'es, Spain} \email{lago@math.uc3m.es}

\keywords{Montessus de Ballore Theorem, simultaneous
approximation, Hermite-Pad\'e approximation}

\subjclass[2010]{Primary 41A21, 41A28; Secondary 41A25}

\begin{abstract} We give a Montessus de Ballore type theorem for row sequences of Hermite-Pad\'e approximations of vector valued analytic functions refining some  results in this direction due to P.R. Graves-Morris and E.B. Saff. We do this introducing the notion of incomplete Pad\'e approximation which contains, in particular, simultaneous Pad\'e approximation and may be applied in the study of other systems of approximants as well.
\end{abstract}

\maketitle
%%%%%%%%%%%%%%%%%%%%%%%%%%%%%%%%%%%%%%%%%%%%%%%%%%%%%%%%%%%%%%%%%%%%%%%%%%%%%%%%%%%%%%%%%%%%%%%%%%%%%%%%%%%%%%%%%%%
%%%%%%%%%%%%%%%%%%%%%%%%%%%%%%%%%%%%%%%%%%%%%%%%%%%%%%%%%%%%%%%%%%%%%%%%%%%%%%%%%%%%%%%%%%%%%%%%%%%%%%%%%%%%%%%
%%%%%%%%%%%%%%%%%%%%%%%%%%%%%%%%%%%%%%      INTRODUCTION   %%%%%%%%%%%%%%%%%%%%%%%%%%%%%%%%%%%%%%%%%%%%%%%%%%
%%%%%%%%%%%%%%%%%%%%%%%%%%%%%%%%%%%%%%%%%%%%%%%%%%%%%%%%%%%%%%%%%%%%%%%%%%%%%%%%%%%%%%%%%%%%%%%%%%%%%%%%%%%%%%%%%%
%%%%%%%%%%%%%%%%%%%%%%%%%%%%%%%%%%%%%%%%%%%%%%%%%%%%%%%%%%%%%%%%%%%%%%%%%%%%%%%%%%%%%%%%%%%%%%%%%%%%%%%%%%%%%%%%%%%

\section{Introduction}
Let
\begin{equation} \label{i1} f(z) = \sum_{n=0}^{\infty} \phi_n z^n,
\qquad \phi_n \in {\mathbb{C}},
\end{equation}
denote a formal or convergent Taylor expansion about the origin. By
$D_0(f)$ and $R_0(f)$ we denote the disk and radius of convergence,
respectively, of the series \eqref{i1}.  In \cite{Had}, Jacques
Hadamard introduced the notion of $m$th disk of meromorphy $D_m(f)$
of $f$. When $R_0(f) = 0$ this disk is defined to be the empty set.
If $R_0(f) > 0$ then $D_m(f)$ is the largest disk centered at the
origin to which the analytic element $(f,D_0(f))$ can be extended as
a meromorphic function having no more than $m$ poles. Let $R_m(f)$
denote the radius of $D_m(f)$. In the cited paper, Hadamard proves a
beautiful formula which gives the values of the numbers $R_m(f)$ for
all $m \in \mathbb{Z}_+$ using exclusively the data provided by the
Taylor coefficients $\phi_n$. For $m=0$, it reduces to  Cauchy's
formula for the radius of convergence of a Taylor series. Hadamard's
finding is intimately connected with the convergence theory of
Pad\'e approximations.

\begin{defi}\label{defpade}
Let $f$ be the formal expansion \eqref{i1}. Let $n,m \in
{\mathbb{Z}}_+, n \geq m,$ be given. Then, there exist polynomials
$Q, P,$ satisfying
\begin{itemize}
\item[a.1)] $\deg P \leq n-m,\quad \deg Q \leq m,\quad Q \not\equiv 0,$
\item[a.2)] $[Q f - P](z) = A z^{n+1} + \cdots.$
\end{itemize}
Any pair $(Q,P)$ which satisfies $a.1)-a.2)$ defines a unique
rational function $\pi_{n,m} = P /Q $ which is called the Pad\'e
approximation of type $(n,m)$ of $f$.
\end{defi}

We have slightly modified (in an equivalent form) the usual
definition of an $(n,m)$ Pad\'e approximation having in mind the
aims of this paper. Let $\pi_{n,m} = P_{n,m}/Q_{n,m}$ where
$Q_{n,m}$ and $P_{n,m}$ are polynomials obtained cancelling all
common factors and, unless otherwise stated, normalizing $Q_{n,m}$ so
that
\begin{equation} \label{normado} Q_{n,m}(z) = \prod_{|\zeta_{n,k}|
\leq 1} \left(z - \zeta_{n,k}\right) \prod_{|\zeta_{n,k}| > 1}
\left(1- \frac{z}{\zeta_{n,k}}\right).
\end{equation}

Robert de Montessus de Ballore, using Hadamard's work, proved the
following result (see \cite{Mon}). Let $\mathcal{Q}_m(f)$ stand for
the polynomial (properly normalized as in $(\ref{normado})$) whose
zeros are the poles of $f$ in $D_m(f)$ with multiplicity equal to
the order of the corresponding pole. By ${\mathcal{P}}_m(f)$ we
denote this set of poles. Given a compact set ${K} \subset
{\mathbb{C}}$, $\|\cdot\|_{K}$ denotes the sup norm on $K$.
\vspace{0,2cm}

\noindent {\bf Montessus de Ballore Theorem.} {\em Assume that
$R_0(f) > 0$ and that $f$ has exactly $m$ poles in $D_m(f)$
(counting multiplicities), then
\begin{equation} \label{eq:6} \limsup_{n \to \infty}
\|f - \pi_{n,m}\|_{K}^{1/n} = \frac{\|z\|_ {K}}{R_{m}(f)},
\end{equation}
where ${K}$ is any compact subset of $D_m(f) \setminus
{\mathcal{P}}_m(f)$. Additionally
\begin{equation} \label{i2}
\limsup_{n \to \infty} \|\mathcal{Q}_m(f) - Q_{n,m}\|^{1/n} =
\frac{\max \{|\zeta|: \zeta \in {\mathcal{P}}_m(f)\}}{R_m(f)},
\end{equation}
where $\|\cdot\|$ denotes the coefficient norm in the space of
polynomials. } \vspace{0,2cm}

From this result it follows that if $\zeta$ is a pole of $f$ in
$D_m(f)$ of order $\tau$, then for each $\varepsilon >0$, there
exists $n_0$ such that for $n \geq n_0$, $Q_{n,m}$ has exactly
$\tau$ zeros in $\{z: |z - \zeta| < \varepsilon\}$. We say that each
pole of $f$ in $D_m(f)$ attracts as many zeros of $Q_{n,m}$ as its
order when $n$ tends to infinity. In Montessus' paper the geometric
rate expressed in (\ref{eq:6}) and (\ref{i2}) was not
explicitly given.

The simultaneous Hermite-Pad\'e approximation of systems of
functions has been a subject of major interest in the recent past.
Though most results deal with what could be called diagonal
sequences of simultaneous approximants, there are some interesting
results due to P. R. Graves-Morris and E. B. Saff for row sequences
which extend the Montessus Theorem, see \cite{GS1}-\cite{GS3}.

\begin{defi}\label{defsimultaneos}
Let ${\bf f} = (f_1,\ldots,f_d)$ be a system of $d$ formal Taylor
expansions as in $(\ref{i1})$. Fix a multi-index ${\bf m} =
({m_1},\ldots,m_d) \in {\mathbb{Z}}_+^d \setminus \{\bf 0\}$ where
${\bf 0}$ denotes the zero vector in ${\mathbb{Z}}_+^d$. Set $|{\bf
m}| = {m_1} +\cdots + m_d$. Then, for each $n \geq \max
\{{m_1},\ldots,m_d\}$, there exist polynomials $Q, P_j,
j=1,\ldots,d,$ such that
\begin{itemize}
\item[b.1)] $\deg P_j \leq n - m_j, j=1,\ldots,d,\quad \deg Q
\leq |\mathbf{m}|,\quad Q \not\equiv 0,$
\item[b.2)] $[Q f_j - P_j ](z) = A_j z^{n+1} + \cdots .$
\end{itemize}
The vector rational function ${\bf R}_{n,{\bf m}} = (P_1 /Q
,\ldots,P_d /Q)$ is called an $(n,{\bf m})$ Hermite-Pad\'e
approximation of ${\bf f}$.
\end{defi}

 Unlike the scalar case, in general, this
vector rational approximation is not uniquely determined and in the
sequel we assume that given $(n,{\bf m})$ one particular solution is
taken. For that solution we write
\begin{equation} \label{incomplete}
{\bf R}_{n,{\bf m}} =
(R_{n,{\bf m},1},\ldots,R_{n,{\bf m},d})= (P_{n,{\bf
m},1},\ldots,P_{n,{\bf m},d})/Q_{n,{\bf m}},
\end{equation}
where $Q_{n,{\bf m}}$ has no common factor simultaneously with all the
$P_{n,{\bf m},j}$ and is normalized the same way as $Q_{n,m}$ above.

\begin{defi}\label{defindependenciapolar1}
A vector ${\bf f} = (f_1,\ldots,f_d)$ of functions meromorphic in
some domain $D$ is said to be polewise independent with respect to
the multi-index ${\bf m} = ({m_1},\ldots,m_d) \in {\mathbb{Z}}_+^d
\setminus \{\bf 0\}$ in $D$ if there do not exist polynomials
$p_1,\ldots,p_d$, at least one of which is non-null, satisfying
\begin{itemize}
\item[c.1)] $\deg p_j \leq m_j -1, j=1,\ldots,d,$ if $m_j \geq 1,$
\item[c.2)] $p_j \equiv 0$ if $m_j =0,$
\item[c.3)] $\sum_{j=0}^d p_j f_j \in {\mathcal{H}}(D)$,
\end{itemize}
where ${\mathcal{H}}(D)$ denotes the space of analytic functions in
$D$.
\end{defi}
This notion was introduced in \cite{GS1}. When $d=1$ polewise
independence merely expresses  that the function has at least $m$ poles in
$D$.

\begin{defi}\label{defpolo}
Let ${\bf f} = (f_1,\ldots,f_d)$ be a system of formal Taylor
expansions about the origin and
$\mathbf{D}=\left(D_1,\dots,D_d\right)$ a system of domains such
that, for each $k=1,\dots,d,$ $f_k$ is meromorphic in $D_k$. We say
that a point a  is a pole of ${\bf f}$ in $\mathbf{D}$ of order
$\tau$ if there exists an index $k\in\{1,\dots,d\}$  such that $a\in
D_k$ and it is a pole of $f_k$ of order $\tau$, and for the rest of the
indices $j\not = k$ either a is a pole of $f_j$ of order less than
or equal to $\tau$ or $a\not\in D_j$.
\end{defi}

%\begin{defi}\label{defpolos0}
% A point $a$ is said to be a pole of ${\bf f} =
%(f_1,\ldots,f_d)$ of order $\tau$ if it is a pole of that order for
%at least one of the functions $\{f_k\}_{k=1}^d$  and it is of order
%less than or equal to $\tau$ for the rest, including the case that
%it may be a point of analyticity for some of the functions.
%\end{defi}

Polewise independence of $\mathbf{f}$ with respect to $\mathbf{m}$
in $D$  implies that $\mathbf{f}$ has at least $|\mathbf{m}|$ poles
in $\mathbf{D}={(D,\dots,D)}$ counting multiplicities, see Lemma 1 in \cite{GS1}.
In those cases when $\mathbf{D}={(D,\dots,D)}$ we say that  a  is a pole of ${\bf f}$ in $D$.

Let $R_0({\bf f})$ be the largest disk in which all the expansions
$f_j, j=1,\ldots,d$ correspond to analytic functions. If $R_0({\bf
f}) =0$, we take $D_{m}({\bf f}) = \emptyset, m \in {\mathbb{Z}}_+$;
otherwise, $R_m({\bf f})$ is the radius of the largest disk
$D_{m}({\bf f})$ centered at the origin to which all the analytic
elements $(f_j, D_0(f_j))$ can be extended so that ${\bf f}$ has at
most $m$ poles counting multiplicities. By
$\mathcal{Q}_m(\mathbf{f})$ we denote the polynomial whose zeros are
the poles of ${\bf f}$ in $D_{m}({\bf f})$ counting multiplicities
and normalized as $Q_{n,{\bf m}}$. This set of poles is denoted by
${\mathcal{P}}_m({\bf f})$.

In \cite{GS1}, Graves-Morris and Saff proved the following analog of
the Montessus de Ballore Theorem for simultaneous approximation.
\vspace{0,2cm}

\noindent {\bf Graves-Morris/Saff Theorem.} {\em Assume that
$R_0({\bf f}) > 0$. Fix a multi-index ${\bf m}\in {\mathbb{Z}}_+^d
\setminus \{\bf 0\}$ and suppose that ${\bf f}$ is polewise
independent with respect to ${\bf m}$ in $D_{|{\bf m}|}({\bf f}), $
then
\begin{equation}\label{inequality3}
 \limsup_{n \to \infty}
\|f_k - R_{n,{\bf m},k}\|_{K}^{1/n} \leq \frac{\|z\|_ {K}}{R_{|{\bf
m}|}({\bf f})}, \quad k=1,\ldots,d,
\end{equation}
where ${K}$ is any compact subset of $ D_{|{\bf m}|}({\bf f})
\setminus {\mathcal{P}}_{|{\bf m}|}({\bf f})$. Additionally
\begin{equation}\label{inequality4}
\limsup_{n \to \infty} \|\mathcal{Q}_{|\mathbf{m}|}(\mathbf{f}) -
Q_{n,{\bf m}}\|^{1/n} \leq \frac{\max\{|\zeta|: \zeta \in
{\mathcal{P}}_{|{\bf m}|}({\bf f})\}}{R_{|{\bf m}|}({\bf f})}.
\end{equation}
}

 It also follows from this result that each pole of ${\bf
f}$ in $D_{|{\bf m}|}({\bf f})$ attracts exactly as many zeros of
$Q_{n,\mathbf{m}}$ as its order when $n$ tends to infinity.

The aim of this paper is to complement and refine some of the statements of
the Graves-Morris/Saff Theorem. For this purpose, in Section 3 we introduce the notion of incomplete Pad\'e approximation and study some of its properties. They are used in Section 4 to obtain our results for row sequences of Hermite-Pad\'e approximation. Section 5 contains examples that illustrate to what extent our main Theorem \ref{saff2} improves the one due to Graves-Morris/Saff. In passing, we mention that incomplete Pad\'e approximants may be used  to study the convergence of other systems of approximating rational (scalar or vector) functions. Section \ref{contenidos} contains some auxiliary results. Our approach is strongly influenced by
the viewpoint of A.A. Gonchar as presented in \cite{gon2} for studying row sequences of Pad\'e approximants.

%%%%%%%%%%%%%%%%%%%%%%%%%%%%%%%%%%%%%%%%%%%%%%%%%%%%%%%%%%%%%%%%%%%%%%%%%%%%%%%%%%%%%%%%%%%%%%%%%%%%%%%%%%%%%%%%%%%
%%%%%%%%%%%%%%%%%%%%%%%%%%%%%%%%%%%%%%%%%%%%%%%%%%%%%%%%%%%%%%%%%%%%%%%%%%%%%%%%%%%%%%%%%%%%%%%%%%%%%%%%%%%%%%%
%%%%%%%%%%%%%%%%%%%%%%%%%%%%%%%%%%%%%%      AUXILIARY RESULTS   %%%%%%%%%%%%%%%%%%%%%%%%%%%%%%%%%%%%%%%%%%%%%%%%%%
%%%%%%%%%%%%%%%%%%%%%%%%%%%%%%%%%%%%%%%%%%%%%%%%%%%%%%%%%%%%%%%%%%%%%%%%%%%%%%%%%%%%%%%%%%%%%%%%%%%%%%%%%%%%%%%%%%
%%%%%%%%%%%%%%%%%%%%%%%%%%%%%%%%%%%%%%%%%%%%%%%%%%%%%%%%%%%%%%%%%%%%%%%%%%%%%%%%%%%%%%%%%%%%%%%%%%%%%%%%%%%%%%%%%%%

\section{Convergence in $\sigma$-content}\label{contenidos} Let $B$ be a subset of the complex
plane $\mathbb{C}$. By $\mathcal{U}(B)$ we denote the class of all
coverings of $B$ by at most a numerable set of disks. Set
$$
\sigma(B)=\inf\left\{\sum_{i=1}^\infty
|U_i|\,:\,\{U_i\}\in\mathcal{U}(B)\right\},
$$
where $|U_i|$ stands for the radius of the disk $U_i$. The quantity
$\sigma(B)$ is called the $1$-dimensional Hausdorff content of the
set $B$. This set function is not a measure but it is  semi-additive
and monotonic, properties which will be used later. Clearly, if $B$
is a disk then $\sigma(B)=|B|$.
\begin{defi}\label{defcontenido}
Let $\{\varphi_n\}_{n\in\mathbb{N}}$ be a sequence of functions
defined on a domain $D\subset\mathbb{C}$ and $\varphi$ another
function defined on $D$. We say that
$\{\varphi_n\}_{n\in\mathbb{N}}$ converges in $\sigma$-content to
the function $\varphi$ in compact subsets of $D$ if for each compact
subset $K$ of $D$ and for each $\varepsilon
>0$, we have
$$
\lim_{n\to\infty} \sigma\{z\in K :
|\varphi_n(z)-\varphi(z)|>\varepsilon\}=0.
$$
Such a convergence will be denoted by $\sigma$-$\lim_{n\to\infty}
\varphi_n = \varphi$ in $D$.
\end{defi}

The next lemma was proved by A.A. Gonchar in \cite{gon}.

\vspace{0,2cm} \noindent {\bf Gonchar's Lemma. }{\em Suppose that
$\sigma$-$\lim_{n\to\infty} \varphi_n = \varphi$ in $D$. Then the
following assertions hold true:
\begin{itemize}

\item[i)] If the functions $\varphi_n,\,n\in\mathbb{N}$,
are holomorphic in $D$, then the sequence $\{\varphi_n\}$ converges
uniformly on compact subsets of $D$ and $\varphi$ is holomorphic in
$D$ (more precisely, it is equal to a holomorphic function in $D$
except on a set of $\sigma$-content zero).

\item[ii)] If each of the functions $\varphi_n$ is meromorphic in $D$
and has no more than $k<+\infty$ poles in this domain, then the
limit function $\varphi$ is  (again except on a set of
$\sigma$-content zero) also meromorphic and has no more than $k$
poles in $D$.

\item[iii)] If each function $\varphi_n$ is meromorphic and has no more than $k<+\infty$ poles
 in $D$ and the function $\varphi$ is meromorphic and has exactly $k$ poles in
 $D$, then all $\varphi_n,\,n\ge N$, also have $k$ poles in $D$; the poles of $\varphi_n$
 tend to the poles $z_1,,\dots,z_k$ of $\varphi$ (taking account of their orders) and
 the sequence $\{\varphi_n\}$ tends to $\varphi$ uniformly on compact subsets of
 the domain $D^{\prime}=D\setminus \{z_1,,\dots,z_k\}$.
\end{itemize}
}

%%%%%%%%%%%%%%%%%%%%%%%%%%%%%%%%%%%%%%%%%%%%%%%%%%%%%%%%%%%%%%%%%%%%%%%%%%%%%%%%%%%%%%%%%%%%%%%%%%%%%%%%%%%%%%%%%%%
%%%%%%%%%%%%%%%%%%%%%%%%%%%%%%%%%%%%%%%%%%%%%%%%%%%%%%%%%%%%%%%%%%%%%%%%%%%%%%%%%%%%%%%%%%%%%%%%%%%%%%%%%%%%%%%
%%%%%%%%%%%%%%%%%%%%%%%%%%%%%%%%%%%%%%      GENERALIZED   %%%%%%%%%%%%%%%%%%%%%%%%%%%%%%%%%%%%%%%%%%%%%%%%%%
%%%%%%%%%%%%%%%%%%%%%%%%%%%%%%%%%%%%%%%%%%%%%%%%%%%%%%%%%%%%%%%%%%%%%%%%%%%%%%%%%%%%%%%%%%%%%%%%%%%%%%%%%%%%%%%%%%
%%%%%%%%%%%%%%%%%%%%%%%%%%%%%%%%%%%%%%%%%%%%%%%%%%%%%%%%%%%%%%%%%%%%%%%%%%%%%%%%%%%%%%%%%%%%%%%%%%%%%%%%%%%%%%%%%%%

\section{Incomplete Pad\'e approximants}\label{incompletos}

\begin{defi}\label{defincompletos}
Let $f$  denote a formal Taylor expansion about the origin. Fix
${m^*} \leq m$. Let $n \geq m$. We say that the rational function
$R_{n,m}$ is an incomplete Pad\'e approximation of type $(n,m,{m^*})
$ corresponding to $f$ if $R_{n,m}$ is the quotient of any two
polynomials $P $ and $Q $ that verify
\begin{itemize}
\item[d.1)]
$\deg P \le n-{m^*},\quad \deg Q \le m,\quad Q \not\equiv 0,$
\item[d.2)] $[Q f-P](z)=
A z^{n+1}+ \cdots .$
\end{itemize}
\end{defi}

Notice that given $(n,m,{m^*}), n \geq m \geq {m^*},$ any one of the
Pad\'e approximants $\pi_{n,{m^*}},\ldots,\pi_{n,m}$ is an incomplete Pad\'e approximation of type $(n,m,{m^*})$
of $f$. The so-called Pad\'e-type approximants (see \cite{bre})
where $m-m^*$ zeros of $Q$ are fixed and $m^*$ are left free are
also incomplete Pad\'e approximants. Moreover, from Definition
\ref{defsimultaneos} and (\ref{incomplete}) it follows that
$R_{n,{\bf m},k}, k=1,\ldots,d,$  is an incomplete Pad\'e
approximation of type $(n,|\mathbf{m}|,m_k)$ with respect to $f_k$.

Given $n \geq m \geq {m^*},$ $R_{n,m}$ is not unique so  we choose one candidate. As before,
after canceling out common factors between $Q$ and $P$, we write
\[ R_{n,m} = P_{n,m}/Q_{n,m},
\]
where, additionally, $Q_{n,m}$ is normalized as in (\ref{normado}).
Suppose that $Q$ and $P$ have a common zero at $z=0$ of order
$\lambda_n$. From d.1)-d.2) readily follows that
\begin{itemize}
\item[d.3)] $ \deg P_{n,m} \leq n-m^*-\lambda_n, \quad
\deg Q_{n,m} \leq m-\lambda_n, \quad Q_{n,m} \not\equiv 0,$
\item[d.4)] $[Q_{n,m} f-P_{n,m}](z)=
A z^{n+1-\lambda_n}+ \cdots .$
\end{itemize}
where $A$ is, in general, a different constant from the one in d.2).

When $f$ denotes a convergent series, it is well known by the
specialists that any row sequence $\{\pi_{n,m}\}_{n \geq m},$ where
$m \ge m^*$ is fixed, converges to $f$ in $\sigma$-content in
compact subsets of $D_{{m^*}}(f)$. This is also true for any
sequence of incomplete Pad\'e approximations when $m \geq {m^*}$ is
fixed. Before giving a formal statement of that result, let us introduce some additional definitions.

Take an arbitrary $\varepsilon > 0$ and define the open set
$J_\varepsilon$ as follows. For $n \geq m$, let $J_{n,\varepsilon}$
denote the $\varepsilon/6mn^2$-neighborhood of the set
${\mathcal{P}}_{n,m} = \{\zeta_{n,1}, \ldots, \zeta_{n,m_n}\}$ of
finite zeros of $Q_{n,m}$. If $R_0(f) > 0$, let
$J_{m-1,\varepsilon}$ denote the $\varepsilon/6m$-neighbor--hood of
the set of poles of $f$ in $D_{m}(f)$. Otherwise,
$J_{m-1,\varepsilon}=\emptyset$. Set $J_\varepsilon = \cup_{n \ge m
-1} J_{n,\varepsilon}$. We have $\sigma(J_{\varepsilon}) <
\varepsilon$ and $J_{\varepsilon_1} \subset J_{\varepsilon_2}$ for $
\varepsilon_1 < \varepsilon_2$. For any set $B \subset {\mathbb{C}}$
we put ${B}(\varepsilon) = {B} \setminus J_{\varepsilon} $.

Let $\{\varphi_n\}_{n\in\mathbb{N}}$ be a sequence of functions
defined on a domain $D$ and $\varphi$ another  function also defined
on $D$. Clearly, if $\{\varphi_n\}_{n\in\mathbb{N}}$ converges
uniformly to $\varphi$ on ${K}(\varepsilon)$ for every compact ${K}
\subset D$ and each $\varepsilon >0$, then
$\sigma$-$\lim_{n\to\infty} \varphi_n = \varphi$ in $D$.

Due to the normalization (\ref{normado}), for any compact set $K$ of
$\mathbb{C}$ and for every $\varepsilon >0$, there exist constants
$C_1, C_2$, independent of $n$, such that
\begin{equation} \label{desig1} \|Q_{n,m}\|_{K} < C_1,
\qquad \min_{z \in {K}(\varepsilon)}|Q_{n,m}(z)| > C_2 n^{-2m},
\end{equation}
where the second inequality is meaningful when ${K}(\varepsilon)$ is
a non-empty set.

In the sequel, $C$ will denote positive constants, generally
different, that are independent of $n$ but may depend on all
the other parameters involved in each formula where they appear.

\begin{propo}\label{contenido}
Let $R_0(f) > 0$. Fix $m$ and ${m^*}$ nonnegative integers, $m \geq
{m^*}$. For each $n \geq m$, let $R_{n,m}$ be an incomplete Pad\'e
approximant of type $(n,m,{m^*})$ for $f$. Then
$$
\sigma\mbox{-}\lim_{n\to\infty} R_{n,m}  = f \;\; \mbox{in}\;\;
D_{m^*}(f).
$$
\end{propo}

\begin{proof}
Let $Q_{m^*}$ denote the polynomial $\mathcal{Q}_{m^*}(f)$ normalized to be monic.
Using d.3), we have
$$
\left[Q_{m^*}Q_{n,m}f-Q_{m^*}P_{n,m}\right](z)= A
z^{n+1-\lambda_n}+\cdots,
$$
which implies that
$$
\frac{\left[Q_{m^*}Q_{n,m}f-Q_{m^*}P_{n,m}\right](z)}{z^{n+1-\lambda_n}}
\in {\mathcal{H}}(D_{m^*}(f)).
$$
Set $|z|<r<R_{m^*}(f)$ with $r$ arbitrarily close to $R_{m^*}(f)$
and let $\Gamma_r=\{z\in\mathbb{C}\,:\,|z|=r\}$. By  Cauchy's
integral formula we obtain
\begin{equation}\label{cauchy}
\begin{array}{l}
\dst \frac{\left[Q_{m^*}Q_{n,m}f-Q_{m^*}P_{n,m}\right](z)}
{z^{n+1-\lambda_n}} = \frac{1}{2\pi i}\int_{\Gamma_r}\frac{[Q_{m^*}
Q_{n,m} f](\zeta)}{\zeta^{n+1-\lambda_n}}\frac{d\zeta}{\zeta-z}+\\ \\
\dst - \int_{\Gamma_r}
\frac{Q_{m^*}(\zeta)P_{n,m}(\zeta)}{\zeta^{n+1-\lambda_n}}
\frac{d\zeta}{\zeta-z}  =\frac{1}{2\pi
i}\int_{\Gamma_r}\frac{[Q_{m^*} Q_{n,m}
f](\zeta)}{\zeta^{n+1-\lambda_n}}\frac{d\zeta}{\zeta-z},
\end{array}
\end{equation}
where the second integral after the first equality is zero due to
the fact that the integrand is an analytic function outside
$\Gamma_r$ with a zero of multiplicity at least two at infinity (see
d.3)).

Fix an arbitrary compact set $K\subset D_{m^*}(f)$ and take
$0<r<R_{m^*}(f)$ such that $K$ and all of the poles of $f$ are
contained in the disk $\{z\in\mathbb{C}\,:\, |z|<r\}$. We also
select an arbitrarily small $\varepsilon>0$.  From \eqref{cauchy} it
follows that
$$
[Q_{m^*}\left(f -R_{n,m}\right)](z)=\frac{z^{n+1-\lambda_n}}{2\pi
i}\int_{\Gamma_r}\frac{[Q_{m^*} Q_{n,m}
f](\zeta)}{Q_{n,m}(z)\,\zeta^{n+1-\lambda_n}}\frac{d\zeta}{\zeta-z},
$$
for all $z\in K(\varepsilon)$. Using this last formula,
\eqref{desig1}, and the continuity of $Q_{m^*}f$ on $\Gamma_r$, we
obtain
$$
\|Q_{m^*} \left(f-R_{n,m}\right)\|_{K(\varepsilon)}\le
C\,\frac{\|z\|_{K}^{n}}{r^n} \frac{\|Q_{n,m}\|_{K}}{\dst\min_{\zeta
\in {K}(\varepsilon)}|Q_{n,m}(\zeta)|}\le
C\,\frac{\|z\|_{K}^{n}}{r^n}\,n^{2m}.
$$
Taking $n$-th root, making $n$ tend to infinity, and letting
$r$ approach $R_{m^*}(f)$, we arrive at
\[
\limsup_{n\to\infty}\|Q_{m^*}
\left(f-R_{n,m}\right)\|^{1/n}_{K(\varepsilon)}\le
\frac{\|z\|_K}{R_{m^*}(f)}<1.
\]
As $\varepsilon>0$ is arbitrary, we have proved
$\sigma$-$\lim_{n\to\infty} Q_{m^*}R_{n,m} = Q_{m^*}f$ in
$D_{m^*}(f)$, which is equivalent to the statement we wanted to
prove.
\end{proof}
Let us find the radius of the largest disk centered at the origin in
compact subsets of which the sequence $\{R_{n,m}\}_{n \geq m}$
converges to $f$ in $\sigma$-content. This number, which depends on
the specific sequence of incomplete Pad\'e approximants considered,
lies between $R_{m^*}(f)$ and $R_{m}(f)$ (see Section \ref{ejemplo1}
below). We need some formulas.

\begin{lem} \label{telescopica}
Let a formal power series $(\ref{i1})$ be given. Fix $m \ge m^*$ two
positive integers. Consider a corresponding sequence of incomplete
Pad\'e approximations. For each $n \ge m$, we have
\[ R_{n+1,m}(z) - R_{n,m}(z) = \frac{A_{n,m} z^{n+1-\lambda_n-
\lambda_{n+1}}q_{n,m-m^*}(z)}{Q_{n,m}(z)Q_{n+1,m}(z)},
\]
where $A_{n,m}$ is some constant and $q_{n,m-m^*}$ is a polynomial
of degree less than or equal to $m - m^*$ normalized as in
$(\ref{normado})$.
\end{lem}

\begin{proof} Using d.4) we have
\[ z^{\lambda_n}[Q_{n,m}f - P_{n,m}](z) = A z^{n+1 } + \cdots
\]
and
\[ z^{\lambda_{n+1}}[Q_{n+1,m}f - P_{n+1,m}](z) = A^\prime z^{n+2} + \cdots.
\]
Multiplying the first equation by $z^{\lambda_{n+1}}Q_{n+1,m}$, the
second by $z^{\lambda_n}Q_{n,m}$, and deleting one of the equations
so obtained from the other, it follows that
\[ z^{\lambda_n +\lambda_{n+1}}[Q_{n,m}P_{n+1,m} - Q_{n+1,m}P_{n,m}](z) =
B z^{n+1} + \cdots.
\]
Taking into consideration d.3) we see that on the left-hand side we
have a polynomial of degree $ \leq n+1 + m-m^*$. Consequently,
\[ z^{\lambda_n +\lambda_{n+1}}[Q_{n,m}P_{n+1,m} - Q_{n+1,m}P_{n,m}](z) =
z^{n+1} \widetilde{q}_{n,m-m^*},
\]
where $\deg \widetilde{q}_{n,m-m^*} \leq m-m^* $. Dividing by
$z^{\lambda_n +\lambda_{n+1}}Q_{n,m}Q_{n+1,m}$ and normalizing
$\widetilde{q}_{n,m-m^*}$ as in (\ref{normado}) we obtain the
desired formula.
\end{proof}

Take an arbitrary $\varepsilon > 0$ and define the open set
$J_\varepsilon^{\prime}$ as follows.  For $n \geq m$, let
$J_{n,\varepsilon}^{\prime}$ denote the
$\varepsilon/6mn^2$-neighborhood of the set of zeros of
$q_{n,m-m^*}$. Set $J_\varepsilon^{\prime} = \cup_{n \ge m}
J_{n,\varepsilon}^{\prime}$. For any compact set ${K} \subset
{\mathbb{C}}$ we put ${K}^{\prime}(\varepsilon) = {K} \setminus
J_{\varepsilon}^{\prime} $.

Due to the fact that the polynomial $q_{n,m-m^*}$ is normalized as
in (\ref{normado}), for any compact set $K$ of $\mathbb{C}$ and for
every $\varepsilon
>0$, there exist constants $M_1, M_2$, independent of $n$, such that
\begin{equation} \label{desig2}
 \|q_{n,m-m^*}\|_{K} < M_1,\qquad\min_{z \in
{K}^{\prime}(\varepsilon)} |q_{n,m-m^*}(z)| > M_2 n^{-2m},
\end{equation}
where the second inequality is meaningful when
${K}^{\prime}(\varepsilon)$ is a non-empty set.

Define
\begin{equation}\label{radio}
{R}_{m}^*(f) =\frac{1}{\dst\limsup_{n\to\infty}
|A_{n,m}|^{1/n}},\qquad {D}_{m}^*(f)=\left\{z\,:\,|z|<{R}_{m}^*(f)
\right\}.
\end{equation}

\begin{thm} \label{teo:5} Let $f$ be a formal power series as in \eqref{i1}. Fix $m$
and ${m^*}$ nonnegative integers, $m \geq {m^*}$. Let
$\{R_{n,m}\}_{n \geq m}$ be a sequence of incomplete Pad\'e
approximants of type $(n,m,{m^*})$ for $f$. If ${R}_{m}^*(f) > 0$
then $R_0(f) > 0$. Moreover,
\[D_{m^*}(f) \subset D_m^*(f) \subset D_m(f)\]
and $D_{m}^*(f)$ is the largest disk in compact subsets of which
$\sigma$-$\lim_{n\to \infty} R_{n,m} = f$. Moreover, the sequence
$\{R_{n,m}\}_{n \geq m}$ is pointwise divergent in $\{z : |z| >
R_{m}^*(f)\}$ except on a set of $\sigma$-content zero.
\end{thm}

\begin{proof} According to Lemma \ref{telescopica}
\begin{equation}\label{telescopic1}
R_{n+1,m}(z) - R_{n,m}(z) = \frac{A_{n,m}
z^{n+1-\lambda_n-\lambda_{n+1}}q_{n,m-m^*}(z)}{Q_{n,m}(z)\,Q_{n+1,m}(z)}.
\end{equation}
Considering telescopic sums, it follows that the sequence
$\{R_{n,m}\}_{n \geq m}$ converges or diverges with the series
$$
\sum_{n \geq n_0} \frac{A_{n,m} z^{n+1-\lambda_n-
\lambda_{n+1}}q_{n,m-m^*}(z)}{Q_{n,m}(z)\,Q_{n+1,m}(z)},
$$
where $n_0$ is chosen conveniently so that $Q_{n_0,m}(z) \neq 0$ at
the specific point under consideration.

Let $R_{m}^*(f) > 0$ and ${K} \subset D_{m}^*(f)$. Fix $\varepsilon
> 0$. Using (\ref{desig1}) and (\ref{desig2}), we have
\begin{equation}\label{telescopic2}
\limsup_{n \to \infty} \left\| \frac{A_{n,m}
z^{n+1-\lambda_n-\lambda_{n+1}}
q_{n,m-m^*}(z)}{Q_{n,m}(z)\,Q_{n+1,m}(z)}
\right\|_{{K}(\varepsilon)}^{1/n} \leq
\frac{\|z\|_{{K}}}{R_{m}^*(f)} < 1.
\end{equation}
Therefore, the series converges uniformly on ${K}(\varepsilon)$ for
every ${K} \subset D_{m}^*(f)$ and every $\varepsilon >0$. Thus $
\sigma$-$\lim_{n \to \infty} R_{n,m} = \varphi$ in $ D_{m}^*(f)$,
where, according to Gonchar's Lemma, $\varphi$ is (except on a set
of $\sigma$-content zero) a meromorphic function with at
most $m$ poles in $D_m^*(f)$. On the other hand, if $|z| >
R_{m}^*(f)$ and $z \not\in J_{\varepsilon}^{\prime}$ from
(\ref{desig1}) and (\ref{desig2}) it follows that
\[ \limsup_{n \to \infty} \left| \frac{A_{n,m} z^{n+1-\lambda_n-
\lambda_{n+1}}q_{n,m-m^*}(z)}{Q_{n,m}(z)\,Q_{n+1,m}(z)}
\right|^{1/n} \geq \frac{|z|}{R_{m}^*(f)} > 1,
\]
and the series diverges. Thus, the sequence $\{R_{n,m}\}_{n
\geq m^*}$ pointwise diverges in $\{z : |z| > R_{m}^*(f)\}$ except
on a set of $\sigma$-content zero (namely, $\cap_{\varepsilon >0} J'_{\varepsilon}$).

We conclude the proof of the theorem if we show that  $R_m^*(f) > 0$
implies that $R_{0}(f) > 0$. Indeed, if this is true, then
necessarily $\varphi = f$ in $D_{m}^*(f)$ since by Proposition
\ref{contenido}, $f$ is the $\sigma$-limit of $\{R_{n,m}\}_{n \geq
m}$ at least in compact subsets of $D_{m^*}(f)$. Since $D_{m}^*(f)$
is the largest disk centered at the origin in compact subsets of
which $\{R_{n,m}\}_{n\geq m}$ converges to $f$ in $\sigma$-content,
we get that $D_{m^*}(f) \subset D_{m}^*(f)$. On the other hand,
$D_{m}(f)$ is the largest disk centered at the origin in which $f$
admits a meromorphic extension with no more than $m$ poles,
therefore $D_{m}^*(f) \subset D_{m}(f)$.

Let $R_m^*(f) >0$, then $\sigma$-$\lim_{n \to \infty} R_{n,m} =
\varphi$ in $D_{m}^*(f)$, where $\varphi$ has at most $m$ poles in
this disk. Choose a subsequence of indices $\Lambda \subset
{\mathbb{N}}$ such that for all $n \in \Lambda$ the number of poles
of $R_{n,m}$ is exactly equal to $m_0,\, m_0\leq m$, and $\lim_{n
\in \Lambda} \zeta_{n,j} = z_j, j=1,\ldots,m_0$. Suppose that $\ell$
of the points $z_j$ equal zero and let $U$ be a neighborhood of
$z=0$ that does not contain any $z_j$ other than zero and is
contained in $D_m^*(f)$. From Gonchar's Lemma it follows that
$\lim_{n \in \Lambda} R_{n,m} = \varphi$ uniformly on each compact
subset of $U^* = U \setminus \{0\}$, where $\varphi$ is holomorphic
in $U^*$, and its Laurent expansion in $U^*$ has the form
\[ \varphi(z) = \sum_{k = - \ell}^{\infty} \varphi_k z^k.
\]
If we show that $\varphi_k =0, k=-\ell,\ldots,-1$, and $\varphi_k =
\phi_k , k \geq 0$, we obtain that $\varphi$ is analytic in $U$ and
coincides with  $f$ in that set. In consequence, $R_0(f) >0$.

Choose $r > 0$ such that $\Gamma = \{z : |z|= r\}$ belongs to $U^*$.
For all sufficiently large $n \in \Lambda$ the points $\zeta_{n,j},
j=1,\ldots,\ell,$ are inside $\Gamma$ and the points $\zeta_{n,j},
j=\ell +1,\ldots,m^*,$ are outside this curve. From now on we only
consider such $n$'s. Let us compare the Taylor expansion of
$R_{n,m}$ about $z=0$
\[ R_{n,m}(z) = \sum_{k=0}^{\infty} \alpha_{n,k}z^k,
\]
with its Laurent expansion on $\Gamma$,
\[ R_{n,m}(z) = \sum_{k=-\infty}^{\infty} \beta_{n,k} z^k.
\]
For notational convenience we set $\phi_k = 0$ and $\alpha_{n,k} =0$
for $k=-1,-2,\ldots$ and $\varphi_k =0$ for $k = -\ell
-1,-\ell-2,\ldots$ We restrict our attention to the case when all
$\zeta_{n,k},\, k=1,\ldots,\ell$, are distinct. The general case is
proved analogously with some additional technical difficulties.

Let $c_{n,j}, j = 1,\ldots,\ell,$ be the residue of $R_{n,m}$ at
$\zeta_{n,j}$. The Taylor expansion of $R_{n,m}$ about $z=0$ and its
Laurent expansion on $\Gamma$ differ only because of the expansion
of the fractions $c_{n,j}/(z-\zeta_{n,j}), j=1,\ldots,\ell$.
Therefore, it is easy to verify that
\begin{equation} \label{eq:34} \beta_{n,k} - \alpha_{n,k} =
\sum_{j=1}^{\ell} \frac{c_{n,j}}{\zeta_{n,j}^{k+1}}, \qquad k \in
\mathbb{Z}.
\end{equation}
By the definition of $R_{n,m}$ (in particular, see d.4)),
$\alpha_{n,k} = \phi_k$ for $k < n+m-\lambda_n$; therefore, $\lim_{n
\in \Lambda} \alpha_{n,k} = \phi_k, k \in {\mathbb{Z}}$. On the
other hand, from the uniform convergence of $R_{n,m}$ to $\varphi$
on $\Gamma$ we also have that $\lim_{n \in \Lambda}\beta_{n,k} =
\varphi_k, k \in {\mathbb{Z}}$. We obtain
\begin{equation} \label{eq:35}
\lim_{n \in \Lambda} (\beta_{n,k} - \alpha_{n,k}) = \varphi_k -
\phi_k, \qquad k \in {\mathbb{Z}}.
\end{equation}
Set $ \varepsilon_{n,k}= \beta_{n,k} - \alpha_{n,k}
$
and
\[ L_{n}(z)= \prod_{j=1}^{\ell} (1 - \zeta_{n,j}z) = 1 + \gamma_{n,1}z +
\cdots+ \gamma_{n,\ell} z^{\ell}.
\]
Using (\ref{eq:34}), for arbitrary $k \in {\mathbb{Z}}$, we obtain
\begin{equation} \label{eq:36}
\varepsilon_{n,k} + \gamma_{n,1}\varepsilon_{n,k+1} +\cdots+
\gamma_{n,\ell} \varepsilon_{n,k+\ell} = \sum_{j=1}^{\ell}
\frac{c_{n,j}}{\zeta_{n,j}^{k+1}} L_{n}(\zeta_{n,j}^{-1}) =0.
\end{equation}
Since $\lim_{n \in \Lambda} \gamma_{n,j} = 0, j=1,\ldots,\ell,$ and
$\lim_{n \in \Lambda} \varepsilon_{n,k+j} = \varphi_{k+j} -
\phi_{k+j}, j=1,\ldots,\ell$, from (\ref{eq:36}) it follows that
$\lim_{n \to \infty} \varepsilon_{n,k} =0$. Using (\ref{eq:35}) we
obtain that $\varphi_k = \phi_k, k \in {\mathbb{Z}},$ as we wanted
to prove.
\end{proof}

Next, we will prove that each pole of the function $f$ in
$D^*_m(f)$ attracts, with geometric rate, at least as many zeros of
$Q_{n,m}$ as its order. For this purpose, let us define two indicators of the
asymptotic behavior of the poles of the incomplete Pad\'e
approximants. These indicators  were first introduced by A.A.
Gonchar in \cite{gon2} for the study of inverse type theorems for
row sequences of Pad\'e approximants. Let
\[ {\mathcal{P}}_{n,m} = \{\zeta_{n,1},\ldots,\zeta_{n,\nu_n}\},
\quad n \in {\mathbb{N}}, \quad \nu_n \leq m,
\]
denote the  collection of zeros of $Q_{n,m}$ (repeated according to
their multiplicity). It is easy to verify that $|\cdot|_1 :
{\mathbb{C}}^2 \longrightarrow {\mathbb{R}}_+$ given by
\[ |z -\omega|_1 = \min\{1, |z -\omega|\}, \qquad z,\omega \in {\mathbb{C}},
\]
defines a distance in $\mathbb{C}$ (although $|\cdot|_1$ does not define
a norm in $\mathbb{C}$).

Choose a point $a \in {\mathbb{C }}$. The first indicator is defined
by
\[
\Delta(a) = \limsup_{n \to \infty} \prod_{j=1}^{\nu_n} |\zeta_{n,j}
- a|_1^{1/n}\, = \limsup_{n \to \infty} \prod_{|\zeta_{n,j} - a| <
1} |\zeta_{n,j} - a|^{1/n}.
\]
Obviously, $0 \leq \Delta(a) \leq 1$ (when $\nu_n =0$ the product is
taken to be $1$). The second indicator, a nonnegative integer
$\mu(a)$, is defined as follows. We suppose that for each $n$ the
points in ${\mathcal{P}}_{n,m}$ are enumerated in nondecreasing distance to the point $a$. We put
\begin{equation} \label{indic2}
\delta_j(a) = \limsup_{n \to \infty} |\zeta_{n,j} -a|_1^{1/n}.
\end{equation}
These numbers are defined by (\ref{indic2}) for $j=1,\ldots,m', m' =
\liminf_{n\to \infty}  \nu_n$; for $j=m'+1,\ldots,m$ we define
$\delta_j(a) =1$. We have $0 \leq \delta_j(a) \leq 1$. If $\Delta(a)
=1$ (in that case all $\delta_j(a) =1$), then $\mu(a) =0$. If
$\Delta(a) < 1$, then for some $\mu, 1 \leq \mu \leq m$, we have
that $\delta_1(a) \leq \cdots \leq \delta_{\mu}(a) < 1$ and
$\delta_{\mu +1}(a) =1$ or $\mu =m$; in this case we take $\mu(a) =
\mu$.

Clearly, $\Delta(a) <1 \Leftrightarrow \mu(a) \geq 1$ and $\sum_{a
\in {\mathbb{C}}} \mu(a)\leq m$. We shall need $\Delta(a)$ and
$\mu(a)$ only for points $a \in {\mathbb{C}}^* = {\mathbb{C}}
\setminus \{0\}$. It is easy to verify that
\begin{equation} \label{eq:Delta} \Delta(a) = \limsup_{n \to \infty}
|Q_{n,m}(a)|^{1/n}.
\end{equation}

\begin{thm}\label{convergence}
Let $R_0(f) > 0$. Fix $m$ and ${m^*}$ nonnegative integers, $m \geq
{m^*}$. For each $n \geq m$, let $R_{n,m}$ be an incomplete Pad\'e
approximant of type $(n,m,{m^*})$ for $f$. Let a be a pole of $f$ in
$D^*_{m}(f)$ of order $\tau$. Then
$$
\Delta(a)\le \frac{|a|}{R^*_{m}(f)}\;\quad \mbox{and}\;\quad \mu(a)
\geq \tau.
$$
\end{thm}
\begin{proof}
Let $a$ be a pole of $f$ in $D^*_{m}(f)$ of order $\tau$ and take
$r>0$ sufficiently small so that the disk of center $a$ and radius
$r$, denoted by $D_{a,r}$, contains no other pole of $f$. It follows
from Gonchar's Lemma that the approximants $R_{n,m}$ have at least
$\tau$ poles in $D_{a,r}$ for sufficiently large $n\in\mathbb{N}$.
If this was not so, from Theorem \ref{teo:5} we have that there exists a
subsequence $\{R_{n,m}\}_{n\in\Lambda}$ converging in
$\sigma$-content to $f$ in compact subsets of $D_{a,r}$ with each
approximant having less than $\tau$ poles in $D_{a,r}$ and part ii) of
Gonchar's Lemma would imply that $f$ has less than $\tau$ poles in
$D_{a,r}$, which is absurd. As $r>0$ is arbitrarily small, we have
proved that each pole of $f$ in $D^*_{m}(f)$ attracts at least as
many zeros of $Q_{n,m}$ as its order.

Fix $\varepsilon>0$ arbitrarily small and take again $r>0$
sufficiently small so that $D_{a,r}$ contains no other pole of $f$.
Since $\sigma(J_{\varepsilon}) < \varepsilon$, we can choose $r$
such that $\Gamma_{a,r}=\{z\,:\, |z-a|=r\}\subset
D_{m^*}(f)\setminus J_\varepsilon$. Let
$\zeta_{n,1},\dots,\zeta_{n,\mu_n}$ be the zeros of $Q_{n,m}$ in
$D_{a,r}$ indexed in non-decreasing distance from $a$. That is,
$$
|a-\zeta_{n,1}|\le|a-\zeta_{n,2}|\le\dots\le |a-\zeta_{n,\mu_n}|.
$$
For all sufficiently large $n$ we know that $\zeta_{n,\tau}\in
D_{a,r}$. We will only consider such $n$'s. Consequently, we have
$\tau\le\mu_n\le m$. Set
$$
Q_{n,a}(z)=\prod_{j=1}^{\mu_n} (z-\zeta_{n,j}).
$$

For any $\rho$ with $|a|+r<\rho<R^*_{m}(f)$, it follows from
\eqref{telescopic1} and \eqref{telescopic2} that
\begin{equation}\label{limit1}
\|f-R_{n,m}\|_{\Gamma_{a,r}}<Cq^n,\qquad q=\frac{|a|+r}{\rho}<1,
\end{equation}
for sufficiently large $n$.

Let $p(z)/(z-a)^\tau$ be the principal part of the function $f$ at
the point $a$ and $p_n/Q_{n,a}$ the sum of the principal parts of
$R_{n,m}$ corresponding to its poles in $D_{a,r}$. We have $\deg
p<\tau,\, p(a)\not =0$, and $\deg p_n<\mu_n$. It is known that the
norm of the holomorphic component of a meromorphic function may be
bounded in terms of the norm of the function and the number of poles
(see Theorem 1 in \cite{gon3}). Thus, using \eqref{limit1}, we
obtain
$$
\left\|
\frac{p(z)}{(z-a)^\tau}-\frac{p_n(z)}{Q_{n,a}(z)}\right\|_{\Gamma_{a,r}}
< C q^n,
$$
for sufficiently large $n$. Therefore, getting rid of the
denominators and applying the maximum principle, we have
\begin{equation}\label{limit4}
\left\| p(z)\,Q_{n,a}(z)-(z-a)^\tau
p_n(z)\right\|_{\overline{D}_{a,r}} < C q^n,
\end{equation}
for sufficiently large $n$. All the factors in $Q_{n,m}$ that
contribute to the limit value $\Delta(a)$ are present in $Q_{n,a}$,
see \eqref{eq:Delta} and \eqref{normado}. So, making $z=a$ in
\eqref{limit4} and taking limits as $n$ tends to infinity gives the
inequality $ \Delta(a)\le q$.  As $r,\,\varepsilon,$ and $\rho$ are
arbitrary we have proved that $\Delta(a)\le |a|/R^*_{m}(f)$. To
conclude the proof we must show that $\mu(a)\ge\tau$. We
will prove it by induction.

Since $\Delta(a)<1$, we have $\delta_1(a)<1$. Let
$\delta_1(a)\le\dots\le\delta_k(a)<1$ and $k<\tau$. We differentiate
the polynomial inside the norm in \eqref{limit4} $k$ times. As this
polynomial has degree bounded by $2m-1$, its $k$th derivative
satisfies an inequality similar to \eqref{limit4} by virtue of
Bernstein's inequality (see, for instance, Section 4.4.2 in
\cite{small}). If we put $z=a$ in the corresponding inequality, we
obtain
\begin{equation}\label{limit3}
\left|\left(p(z)\prod_{j=1}^{\mu_n}(z-\zeta_{n,j})
\right)^{(k)}(a)\right|<Cq^n,
\end{equation}
for sufficiently large $n$. Now
\begin{equation}\label{limit2}
\left(p(z)\prod_{j=1}^{\mu_n}(z-\zeta_{n,j}) \right)^{(k)}(a)
=\sum_{|\alpha|=k}\frac{k!}{\alpha!}\, p^{(\beta)}(a)
\prod_{j=1}^{\mu_n}(z-\zeta_{n,j})^{(\alpha_j)}(a),
\end{equation}
where
$\alpha=(\beta,\alpha_1,\dots,\alpha_{\mu_n})\in\mathbb{Z}_+^{\mu_n+1}$,
$\alpha!=\beta!\cdot\alpha_1!\cdot {\dots} \cdot\alpha_{\mu_n}!\,$,
and $|\alpha|=\beta+\alpha_1+\dots+\alpha_{\mu_n}$. By
$\sum_{|\alpha|=k}$ we mean that the sum is taken over all the
multi-indices $\alpha$ such that $|\alpha|=k$. The total amount of
such multi-indices is bounded independently of $n$. One of them is
$(0,1,\dots,1,0,\dots,0)$ corresponding to the term
$$
k!\, p(a)\prod_{j=k+1}^{\mu_n}(z-\zeta_{n,j}).
$$
Each of the remaining terms must necessarily contain one factor
$(z-\zeta_{n,j}),\, j\in\{1,2,\dots,k\}$. Since we have assumed that
$\delta_j(a)<1$ for $j=1,\dots,k$, it follows from \eqref{limit3}
and \eqref{limit2} that
$$
\limsup_{n\to\infty} \prod_{j=k+1}^{\mu_n}|z-\zeta_{n,j}|^{1/n}<1,
$$
which in turn implies $\limsup_{n\to\infty}
|z-\zeta_{n,k+1}|^{1/n}<1$, that is, $\delta_{k+1}(a)<1$. Therefore
it holds $\mu(a)\ge\tau$ and we are done.
\end{proof}

The estimate $\Delta(a)\le |a|/R^*_{m}(f)$ can be
sharpened if one knows that a given pole attracts exactly as many
zeros of $Q_{n,m}$ as its order.

\begin{thm} \label{teo:4} Let $R_0(f) > 0$ and let $a$ be a pole of
$f$ in $D^*_{m}(f)$ of order $\tau$. Assume that $\liminf_{n \to
\infty} |a - z_{n,\tau+1}| > 0$. Then
\[ \delta_1(a) \leq \cdots \leq \delta_{\tau}(a) \leq
\left(\frac{|a|}{R^*_{m}(f)}\right)^{1/\tau}\!\!.
\]
In particular, $\delta_1(a) = \cdots = \delta_{\tau}(a) = (|a|/R_m^*(f))^{1/\tau}$ if and only if $\Delta(a)= |a|/R_m^*(f)$.
\end{thm}

\begin{proof} Let us maintain the notation used in the proof of
Theorem \ref{convergence}. We may assume that
$$
Q_{n,a}(z)=\prod_{j=1}^{\tau} (z-\zeta_{n,j}).
$$
Recall that $p(a)\not =0$. So, taking $z=a$ in \eqref{limit4}, we
obtain
$
|Q_{n,a}(a)| <Cq^n,
$
for sufficiently large $n$. From this, \eqref{limit3}, and the
formula
$$
\left(p\, Q_{n,a} \right)^{(k)}(a)= p(a)\,Q_{n,a}^{(k)}(a)+
\sum_{j=0}^{k-1}
 \left(\!\!\begin{array}{c}
k\\
j\\
\end{array}\!\!
\right)p^{(k-j)}(a)\, Q_{n,a}^{(j)}(a)
$$
it readily follows by induction that
\begin{equation}\label{derivada1}
 |Q_{n,a}^{(k)}(a)| \leq
Cq^n, \qquad k=0,1,\dots,\tau-1,
\end{equation}
for sufficiently large $n$. These inequalities and the expression
\begin{equation}\label{derivada2}
Q_{n,a}(z)=(z-a)^\tau+\sum_{k=0}^{\tau-1}
\frac{Q_{n,a}^{(k)}(a)}{k!}(z-a)^{k}
\end{equation}
give $
 \left\|(z-a)^\tau-Q_{n,a}(z)\right\|_{\overline{D}_{a,r}}<
Cq^n,
$
for $n\ge N\in\mathbb{N}$. If we put here $z=\zeta_{n,\tau}$ we
obtain
$$
|\zeta_{n,\tau}-a|^\tau< Cq^n,\qquad n\ge N,
$$
which implies $\delta_\tau(a)^\tau\le q$. As $q=(|a|+r)/\rho$ and
$r>0$ and $\rho<R^*_{m}(f)$ are arbitrary, we have
\[ \delta_\tau(a) \leq \left(\frac{|a|}{R^*_{m}(f)}\right)^{1/\tau},
\]
which is all we need to show since $\delta_1(a) \leq \cdots \leq
\delta_{\tau}(a)$ is trivial.

On the other hand, according to Theorem \ref{convergence}, $\Delta(a) \leq \frac{|a|}{R^*_{m}(f)}$ is always true and the last statement readily follows.
\end{proof}

%%%%%%%%%%%%%%%%%%%%%%%%%%%%%%%%%%%%%%%%%%%%%%%%%%%%%%%%%%%%%%%%%%%%%%%%%%%%%%%%%%%%%%%%%%%%%%%%%%%%%%%%%%%%%%%%%%%
%%%%%%%%%%%%%%%%%%%%%%%%%%%%%%%%%%%%%%%%%%%%%%%%%%%%%%%%%%%%%%%%%%%%%%%%%%%%%%%%%%%%%%%%%%%%%%%%%%%%%%%%%%%%%%%
%%%%%%%%%%%%%%%%%%%%%%%%%%%%%%%%%%%%%%      MONTESSUS   %%%%%%%%%%%%%%%%%%%%%%%%%%%%%%%%%%%%%%%%%%%%%%%%%%
%%%%%%%%%%%%%%%%%%%%%%%%%%%%%%%%%%%%%%%%%%%%%%%%%%%%%%%%%%%%%%%%%%%%%%%%%%%%%%%%%%%%%%%%%%%%%%%%%%%%%%%%%%%%%%%%%%
%%%%%%%%%%%%%%%%%%%%%%%%%%%%%%%%%%%%%%%%%%%%%%%%%%%%%%%%%%%%%%%%%%%%%%%%%%%%%%%%%%%%%%%%%%%%%%%%%%%%%%%%%%%%%%%%%%%

\section{Simultaneous approximation}\label{simultaneos}
Throughout this section, $\mathbf{f}=(f_1,\dots,f_d)$ denotes a
system of formal power expansions about the origin;  that is,
$$
f_k(z)=\sum_{j=0}^\infty \phi_{k,j}\,z^{j},\quad k=1,\dots,d,
$$
and
$\mathbf{m}=(m_1,\dots,m_d)\in\mathbb{Z}^d_+\setminus\{\mathbf{0}\}$
is a fixed multi-index. We are concerned with the simultaneous
approximation of $\textbf{f}$ by sequences of vector rational
functions defined according to Definition \ref{defsimultaneos}
taking account of (\ref{incomplete}). That is, for each
$n\in\mathbb{N}, n \geq |{\bf m}|$, let
$\left(R_{n,\mathbf{m},1},\dots,R_{n,\mathbf{m},d}\right)$ be a
Hermite-Pad\'e approximation of type $(n,\mathbf{m})$ corresponding
to $\mathbf{f}$. As we mentioned earlier, $R_{n,\mathbf{m},k}$ is an
incomplete Pad\'e approximant of type $(n,|\mathbf{m}|,m_k)$ with
respect to $f_k,\,k=1,\dots,d$. In the sequel, we consider $\Delta$
and $\mu$ defined as in Section \ref{incompletos} taking
$\mathcal{P}_{n,m}$ to be the collection of zeros of the common
denominator $Q_{n,\mathbf{m}}$.

The number $R_{|\mathbf{m}|}^*(f_k)$ determines the radius of the
largest disk, denoted by $D_{|\mathbf{m}|}^*(f_k)$, centered at the
origin in compact subsets of which we have
$\sigma$-$\lim_{n\to\infty} R_{n,\mathbf{m},k} = f_k$. Theorem
\ref{teo:5} gives $R_{m_k}(f_k)\le R_{|\mathbf{m}|}^*(f_k)\le
R_{|\mathbf{m}|}(f_k)$ and a formula for finding
$R_{|\mathbf{m}|}^*(f_k)$. The following result is a rather
straightforward consequence of Theorem \ref{convergence}.

\begin{cor}\label{convergence2}
Suppose that $R_0(\mathbf{f}) >0$. For each $k=1,\dots,d,$ if a is a pole
of $f_k$ in $D^*_{|\mathbf{m}|}(f_k)$ of order $\tau$, then $
\Delta(a)\le |a|/R^*_{|\mathbf{m}|}(f_k)$ and $\mu(a) \geq \tau$.
\end{cor}
\begin{proof}
Fix $k=1,\dots,d$. Denote the denominator of ${R}_{n,\mathbf{m},k}$,
considered as an incomplete Pad\'e approximant of $f_k$,  by
${Q}_{n,{\bf m},k}$. This polynomial is either $Q_{n,{\bf m}}$ or a divisor
of it, since there may be some additional cancelations of common factors with
the numerator of ${R}_{n,\mathbf{m},k}$. Let ${\Delta}_k$ and
${\mu}_k$ stand for the indicators $\Delta$ and $\mu$, respectively,
when using ${Q}_{n,{\bf m},k}$ instead of $Q_{n,{\bf m}}$. Let $a$ be a pole
of $f_k$ in $D^*_{|\mathbf{m}|}(f_k)$ of order $\tau$. Then, Theorem
\ref{convergence} gives
$$
\Delta_k(a)\le |a|/R^*_{|\mathbf{m}|}(f_k),\qquad \mu_k (a) \geq
\tau.
$$
It is clear that $\mu(a)\ge\mu_k(a)$ whereas
$$
\Delta
(a)=\limsup_{n\to\infty}\left|Q_{n,\mathbf{m}}(a)\right|^{1/n} \le
\limsup_{n\to\infty}\left|{Q}_{n,m}(a)\right|^{1/n}= \Delta_k (a),
$$
which proves the result.
\end{proof}

To each pole $a$ of ${\bf f}$ in a system of domains
$\mathbf{D}=(D_1,\dots,D_d)$ (see Definition \ref{defpolo}) we
associate an index $k(a)\in\{1,\dots,d\}$ as follows. The index
$k(a)$ verifies that $a\in D_{k(a)}$ and $a$ is a pole of $f_{k(a)}$
of the same order as it is as a pole of ${\bf f}$ in $\mathbf{D}$.
If there are several indices $k$ satisfying that condition we choose
one among those with greatest $R^*_{|\mathbf{m}|}(f_k)$.

Given a system ${\bf f} = (f_1,\ldots,f_d)$ and
 a multi-index
$\mathbf{m}\in \mathbb{Z}^d_+ \setminus \{\bf 0\}$, put
$$
\mathbf{D}_{\mathbf{m}}^*(\mathbf{f})=\left(D^*_{|\mathbf{m}|}(f_1),\dots,D^*_{|\mathbf{m}|}(f_d)
\right),\qquad D^*_{\mathbf{m}}(\mathbf{f})=\bigcap_{k=1}^d
D^*_{|\mathbf{m}|}(f_k),
$$
and let $R^*_{\mathbf{m}}(\mathbf{f})$ stand for the radius of
$D^*_{\mathbf{m}}(\mathbf{f})$. By
$\mathcal{Q}_\mathbf{m}(\mathbf{f})$ we denote the polynomial whose
zeros are the poles of ${\bf f}$ in
$\mathbf{D}_{\mathbf{m}}^*(\mathbf{f})$ counting multiplicities and
normalized as $Q_{n,{\bf m}}$ in \eqref{normado}. This set of poles is denoted by
 $\mathcal{P}_{\mathbf{m}}(\mathbf{f})$. For $k=1,\dots,d$, set
 $\mathcal{P}_{\mathbf{m},k}(\mathbf{f})=
 \mathcal{P}_{\mathbf{m}}(\mathbf{f})\cap D^*_{|\mathbf{m}|}(f_k)$.

\begin{lem}\label{lema2}
The following assertions hold:
\begin{itemize}
\item[a)] If $R_0(\mathbf{f}) >0$ then ${\bf f}$ has at most $|\mathbf{m}|$
poles in $\mathbf{D}_{\mathbf{m}}^*(\mathbf{f})$.
\item[b)] If $R^*_{\mathbf{m}}(\mathbf{f})>0$ then $R_0(\mathbf{f}) >0$.
\item[c)] $R^*_{\mathbf{m}}(\mathbf{f})\le
R_{|\mathbf{m}|}(\mathbf{f})$.
\item[d)] Suppose that ${\bf f}$
is polewise independent with respect to ${\bf m}$ in $D_{|{\bf
m}|}({\bf f}), $ then $R^*_{\mathbf{m}}(\mathbf{f})=
R_{|\mathbf{m}|}(\mathbf{f})$.
\end{itemize}
\end{lem}
\begin{proof}
Suppose that $R_0(\mathbf{f}) >0$ and ${\bf f}$ has more than
$|\mathbf{m}|$ poles in $\mathbf{D}_{\mathbf{m}}^*(\mathbf{f})$.
Due to Corollary \ref{convergence2}, each of those poles attracts as many zeros of $Q_{n,\mathbf{m}}$ as
its order. Then, $\deg
Q_{n,\mathbf{m}}>|\mathbf{m}|$, which is absurd. Therefore, a) takes place.

If $R^*_{\mathbf{m}}(\mathbf{f})>0$ then, for each $k=1,\dots,d$,
$R^*_{|\mathbf{m}|}(f_k)>0$. By virtue
of Theorem \ref{teo:5}, taking $m=|\mathbf{m}|$ and $m^*=m_k$, this implies $R_{0}(f_k)>0, k=1,\ldots,d$. This
proves assertion b). As for c), if $R_0(\mathbf{f}) =0$ then the
result is trivial due to part b). Suppose that $R_0(\mathbf{f}) >0$
and $R_{|\mathbf{m}|}(\mathbf{f})<R^*_{\mathbf{m}}(\mathbf{f})$.
Then $\mathbf{f}$ has at least $|\mathbf{m}|+1$ poles in
$D^*_{\mathbf{m}}(\mathbf{f})$, which contradicts part a).

Regarding d) we can assume that $R_0(\mathbf{f}) >0$ since the case
$R_0(\mathbf{f}) =0$ is trivial. The Graves-Morris/Saff Theorem says
in particular that, for each $k=1,\dots,d$, we have
$\sigma$-$\lim_{n\to\infty} R_{n,{\bf m},k} = f_k$ in
$D_{|\mathbf{m}|}(\mathbf{f})$. From the definition of
$D^*_{|\mathbf{m}|}(f_k)$ it follows that
$D_{|\mathbf{m}|}(\mathbf{f})\subset D^*_{|\mathbf{m}|}(f_k),\,
k=1,\dots,d$. Hence $D_{|{\bf m}|}({\bf f}) \subset D_{{\bf m}}({\bf f})$ and by c) the equality follows. With this, we conclude the proof.
\end{proof}

We will use the following concept in the next theorem.

\begin{defi}\label{defregular}
We say that a compact set $K\subset\mathbb{C}$ is $\sigma$-regular
if for each $z_0\in K$ and for each $\delta>0$, it holds
$$
\sigma\{z\in K\,:\, |z-z_0|<\delta \}>0.
$$
\end{defi}

We are ready to prove our main result.

\begin{thm}\label{saff2}
Let $\mathcal{P}_{\mathbf{m}}(\mathbf{f})=\{a_1,\dots,a_\nu\}$.
Suppose that $R_0(\mathbf{f}) >0$ and that $\mathbf{f}$ has exactly
$|\mathbf{m}|$ poles in $\mathbf{D}_{\mathbf{m}}^*(\mathbf{f})$.
Then,
\begin{equation}\label{inequality}
 \limsup_{n \to \infty}
\|f_k - R_{n,{\bf m},k}\|_{K}^{1/n} \leq \frac{\|z\|_
{K}}{R^*_{|{\bf m}|}(f_k)}, \quad k=1,\ldots,d,
\end{equation}
where ${K}$ is any compact subset of $  D^*_{|{\bf m}|}(f_k)
\setminus {\mathcal{P}}_{{\bf m},k}({\bf f})$.  Also, we have
\begin{equation}\label{inequality2}
\limsup_{n \to \infty} \|\mathcal{Q}_{\mathbf{m}}(\mathbf{f}) -
Q_{n,{\bf m}}\|^{1/n}\le \max_{i=1,\dots,\nu}\left\{\frac{|a_i|}{
R^*_{|{\bf m}|}(f_{k(a_i)})}\right\}.
\end{equation}
If, additionally, ${K}$ is $\sigma$-regular, then
we have equality in \eqref{inequality}.
\end{thm}
\begin{proof}
Let $a$ be an arbitrary pole of ${\bf f}$ in
$\mathbf{D}_{\mathbf{m}}^*(\mathbf{f})$ and let $\tau$ be its order.
Then, $a$ is a pole of $f_{k(a)}$ in $ D^*_{|{\bf m}|}(f_{k(a)})$ of
order $\tau$. According to Corollary \ref{convergence2}, we have
$\mu(a)\ge\tau$. As this is true for any other pole of $\mathbf{f}$
in $\mathbf{D}_{\mathbf{m}}^*(\mathbf{f})$ and $\deg
Q_{n,\mathbf{m}}\le |\mathbf{m}|$, we have $\deg Q_{n,\mathbf{m}}=
|\mathbf{m}|$ for sufficiently large $n$, $\mu(a)=\tau$, and
\begin{equation}\label{tiendeacero}
\lim_{n \to \infty} \|\mathcal{Q}_{\mathbf{m}}(\mathbf{f}) -
Q_{n,{\bf m}}\|=0.
\end{equation}

Take $r>0$ sufficiently small so that $D_{a,r}$ contains no other
pole of $\mathbf{f}$. Let $\zeta_{n,1},\dots,\zeta_{n,\mu_n}$ be the
zeros of $Q_{n,\mathbf{m}}$ in $D_{a,r}$ indexed in increasing
distance from $a$. That is,
$$
|a-\zeta_{n,1}|\le|a-\zeta_{n,2}|\le\dots\le |a-\zeta_{n,\mu_n}|.
$$
We know that $\mu_n\ge \tau$ and $\liminf_{n \to \infty} |a -
z_{n,\tau+1}| > 0$, so we can use the arguments employed in Theorem
\ref{teo:4}. In particular, formulas \eqref{derivada1} and
\eqref{derivada2} prove that
\begin{equation}\label{limit5}
\limsup_{n\to\infty} \left\|(z-a)^\tau-Q_{n,a}\right\|^{1/n}\le
\frac{|a|}{R^*_{|{\bf m}|}(f_{k(a)})},
\end{equation}
where
$$ Q_{n,a}(z)=\prod_{j=1}^{\tau} (z-\zeta_{n,j}).
$$
 Formula \eqref{limit5} holds true for each
of the poles of $\mathbf{f}$, so it may be rewritten as
\begin{equation}\label{limit6}
\limsup_{n\to\infty}
\left\|(z-a_i)^{\tau_i}-Q_{n,a_i}\right\|^{1/n}\le
\frac{|a_i|}{R^*_{|{\bf m}|}(f_{k(a_i)})},
\end{equation}
where $\tau_i$ is the order of $a_i$ as a pole of $\mathbf{f}$ in
$\mathbf{D}_{\mathbf{m}}^*(\mathbf{f})$, $i=1,\dots,\nu$.

Let $Q_{|\mathbf{m}|}$ and $Q_{n,|\mathbf{m}|}$ be the polynomials
$\mathcal{Q}_{\mathbf{m}}$ and $Q_{n,\mathbf{m}}$ respectively
normalized to be monic. We can write
$$
\begin{array}{rcl}
\dst \left(Q_{|\mathbf{m}|}-Q_{n,{|\mathbf{m}|}}\right)(z)&=&\dst
Q_{|\mathbf{m}|}(z)-\frac{(Q_{|\mathbf{m}|}Q_{n,a_1})(z)}
{(z-a_1)^{\tau_1}}+\frac{(Q_{|\mathbf{m}|}Q_{n,a_1})(z)}
{(z-a_1)^{\tau_1}}-\cdots\\ \\
\dst&+&\dst \frac{(Q_{|\mathbf{m}|}Q_{n,a_1}\cdots
Q_{n,a_{\nu-1}})(z)}
{(z-a_1)^{\tau_1}\cdots(z-a_{\nu-1})^{\tau_{\nu-1}}}-Q_{n,{|\mathbf{m}|}}(z).
\end{array}
$$
Therefore
$$
\left|Q_{|\mathbf{m}|}-Q_{n,{|\mathbf{m}|}}\right|(z)\le\sum_{i=1}^\nu
\left| \frac{(Q_{|\mathbf{m}|}Q_{n,a_1}\cdots Q_{n,a_{i-1}})(z)}
{(z-a_1)^{\tau_1}\cdots(z-a_{i})^{\tau_{i}}}\left[
(z-a_i)^{\tau_i}-Q_{n,a_i}(z) \right]\right|.
$$
Since
$$
\lim_{n\to\infty}\frac{(Q_{|\mathbf{m}|}Q_{n,a_1}\cdots
Q_{n,a_{i-1}})(z)} {(z-a_1)^{\tau_1}\cdots(z-a_{i})^{\tau_{i}}}=
\frac{Q_{|\mathbf{m}|}(z)}{(z-a_i)^{\tau_i}},\quad i=1,\dots,\nu,
$$
uniformly on compact subsets of $\mathbb{C}$, with the aid of
\eqref{limit6}, we obtain the inequality \eqref{inequality2}.

Now, fix $k=1,\dots,d$. Let ${K}$ be an arbitrary compact subset of
$ D^*_{|{\bf m}|}(f_k) \setminus {\mathcal{P}}_{{\bf m},k}({\bf
f})$. Due to \eqref{tiendeacero}, and reasoning only for sufficiently
large values of $n$, we have that $K=K(\varepsilon)$ for all $\varepsilon > 0$ sufficiently small, where the definition of
$J_\varepsilon$ is given for $Q_{n,\mathbf{m}}$. Then, the
inequality \eqref{inequality} follows from the formulas
\eqref{telescopic1} and \eqref{telescopic2} when applied to the
incomplete Pad\'e approximant $R_{n,{\bf m},k}$.

Suppose now that the compact set ${K}\subset D^*_{|{\bf m}|}(f_k)
\setminus {\mathcal{P}}_{{\bf m},k}({\bf f})$ is $\sigma$-regular,
see Definition \ref{defregular}. Let us consider the constants
$A_{n,{\bf m},k}$ and the polynomials $q_{n,m-m^*,k}$ defined according to
Lemma \ref{telescopica} for the incomplete Pad\'e approximant
$R_{n,{\bf m},k}$, where $m=|\mathbf{m}|$ and $m^*=m_k$. Denote the
denominator of $R_{n,{\bf m},k}$ by $Q_{n,{\bf m},k}$. Put $
J_0^\prime=\cap_{\varepsilon>0} J_\varepsilon^\prime$ and take
$z_0\in K$ such that $\|z\|_K=|z_0|>0$. As $J_0^\prime$ is a set of
$\sigma$-content zero and the compact set $K$ is $\sigma$-regular,
there exists a sequence $\{z_j\}_{j\in\mathbb{N}}\subset K\setminus
J_0^\prime$ verifying $ \lim_{j\to\infty}z_j=z_0. $ We may assume
that $|z_j|>0$ for all $j\in\mathbb{N}$.

From Lemma \ref{telescopica}, it follows that
$$
|A_{n,{\bf m},k}|=\frac{|(Q_{n+1,{\bf m},k}\,Q_{n,{\bf m},k})(z_j)|\,
|R_{n+1,\mathbf{m},k}(z_j)-R_{n,\mathbf{m},k}(z_j)|}
{|z_j|^{n+1-\lambda_n-\lambda_{n+1}}|q_{n,m-m^*,k}(z_j)|}.
$$
We may write
$$
|R_{n+1,\mathbf{m},k}(z_j)-R_{n,\mathbf{m},k}(z_j)|\le
\|f_k-R_{n+1,\mathbf{m},k}\|_K+ \|f_k-R_{n,\mathbf{m},k}\|_K.
$$
So, taking into account the formulas \eqref{desig1} and
\eqref{desig2}, we arrive at
$$
\frac{1}{R^*_{|{\bf
m}|}(f_k)}=\limsup_{n\to\infty}|A_{n,{\bf m},k}|^{1/n}\le
\frac{1}{|z_j|}\dst\limsup_{n\to\infty}\|f_k-R_{n,{\bf m},k}\|^{1/n}_K.
$$
Taking limits in the above expression as $j$ tends to infinity, we
obtain that the inequality \eqref{inequality} is actually an
equality when $K$ is a $\sigma$-regular compact set, as we wanted to
prove.
\end{proof}

As was mentioned earlier,  if $\mathbf{f}$ is polewise independent in
$D_{|\mathbf{m}|}(\mathbf{f})$  it follows from Lemma 1 in
\cite{GS1} that $\mathbf{f}$ has exactly $|\mathbf{m}|$ poles in
$D_{|\mathbf{m}|}(\mathbf{f})$ and, due to part d) of Lemma
\ref{lema2}, it has at least  $|\mathbf{m}|$ poles in
$\mathbf{D}_{\mathbf{m}}^*(\mathbf{f})$. Now, part a) of Lemma
\ref{lema2} proves that Theorem \ref{saff2} includes the
Graves-Morris/Saff Theorem as a particular case although we have
used the latter in establishing this fact.

Theorem \ref{saff2} improves the
Graves-Morris/Saff Theorem in several aspects. First of all, \eqref{inequality} gives the correct bound since we have shown that it is exact for $\sigma$-regular compact sets. The
applicability of  Theorem \ref{saff2} is greater since there are
systems $\mathbf{f}$ that are not polewise independent
 and still  have exactly
$|\mathbf{m}|$ poles in $\mathbf{D}_{\mathbf{m}}^*(\mathbf{f})$.
Even when the system $\mathbf{f}$ is polewise independent the region of
convergence of the approximants given by Theorem \ref{saff2} is in
general larger than that of the Graves-Morris/Saff Theorem. Finally,
the bounds \eqref{inequality} and \eqref{inequality2} are less than or
equal to \eqref{inequality3} and \eqref{inequality4}, respectively.
Several examples in Section \ref{ejemplo3} illustrate these improvements. In Section \ref{ejemplo4} we show that, in general, the bound \eqref{inequality2} is still not exact.

%%%%%%%%%%%%%%%%%%%%%%%%%%%%%%%%%%%%%%%%%%%%%%%%%%%%%%%%%%%%%%%%%%%%%%%%%%%%%%%%%%%%%%%%%%%%%%%%%%%%%%%%%%%%%%%%%%
%%%%%%%%%%%%%%%%%%%%%%%%%%%%%%%%%%%%%%%%%%%%%%%%%%%%%%%%%%%%%%%%%%%%%%%%%%%%%%%%%%%%%%%%%%%%%%%%%%%%%%%%%%%%%%%%%%
%%%%%%%%%%%%%%%%%%%%%%%%%%%%%%%%      EJEMPLOS    %%%%%%%%%%%%%%%%%%%%%%%%%%%%%%%%%%%%%%%%%%%%%
%%%%%%%%%%%%%%%%%%%%%%%%%%%%%%%%%%%%%%%%%%%%%%%%%%%%%%%%%%%%%%%%%%%%%%%%%%%%%%%%%%%%%%%%%%%%%%%%%%%%%%%%%%%%%%%%%%%
%%%%%%%%%%%%%%%%%%%%%%%%%%%%%%%%%%%%%%%%%%%%%%%%%%%%%%%%%%%%%%%%%%%%%%%%%%%%%%%%%%%%%%%%%%%%%%%%%%%%%%%%%%%%%%%%%%

\section{Examples}\label{examples}
\subsection{On the values of $R^*_m(f)$}\label{ejemplo1}

The purpose of this
example is to show that $R^*_m(f)$ may take any value between
$R_{m^*}(f)$ and $R_m(f)$ depending on the sequence of incomplete
Pad\'e approximants considered. Take $m^*=1,\,m=2$, and
$$
 f(z)=\frac{1}{1-z^2}.
$$
Then, $R_1(f)=1,\,R_2(f)=+\infty$. Consider
$$
g(z)=\frac{z}{1+z^2},\quad h(z)=\frac{1}{1+z},\quad
w_p(z)=\frac{1}{1+z}+\frac{1}{1-z/p},\quad p>1.
$$
Fix ${\bf m}=(1,1)$ and set $\mathbf{f}=(f,g)$. It is clear that
$R_2(\mathbf{f})=1$ and the system $\mathbf{f}$ is not polewise
independent with respect to $\mathbf{m}$  in $D_2(\mathbf{f})$. On
the other hand, $R_1(f)=R_1(g)=1$ and $R_2(f)=R_2(g)=+\infty$. It is
very easy to see that $Q_{n,\mathbf{m}}=1-z^2$ if $n$ is even and
$Q_{n,\mathbf{m}}=1+z^2$ when $n$ is odd. So, $R_2^*(f)=1$ since
$R_2^*(f)\ge R_1(f)=1$ and $R_2^*(f)$ cannot be greater than $1$.
Otherwise, from part iii) of Gonchar's Lemma, it follows that the
polynomial $Q_{n,\mathbf{m}}$ tends to $1-z^2$, which is not true.
An analogous argument proves that $R_2^*(g)=1$. This example also
shows that the reciprocal of the statement d) in Lemma \ref{lema2}
does not hold in general.

Now, take $\mathbf{f}=(f,h)$ with the same multi-index ${\bf m}$.
Obviously, $R_2^*(h)=+\infty$ since $R_1(h)=+\infty$. The system
$\mathbf{f}$ is polewise independent with respect to $\mathbf{m}$ in
$D_2(\mathbf{f})=\mathbb{C}$.  Using  part d) of Lemma
\ref{lema2}, we obtain $R_2^*(f)=+\infty$.

Finally, consider $\mathbf{f}=(f,w_p)$ and fix ${\bf m}=(1,1)$. We
have $R_2(\mathbf{f})=p$ and the system $\mathbf{f}$ is  polewise
independent with respect to $\mathbf{m}$  in $D_2(\mathbf{f})$. As
$R_2^*(w_p)\ge R_1(w_p)=p$, necessarily $R_2^*(f)\ge p$ due to Lemma
\ref{lema2} again. Then  $Q_{n,\mathbf{m}}$ tends to $1-z^2$ and
$R_2^*(w_p)=p$. An easy calculation shows that
$$
Q_{n,\mathbf{m}}(z)=\left\{
  \begin{array}{ll}
   \dst \lambda_n\left(z^2+\frac{p^2-1}{p^n-p}z-1\right), & \hbox{if $n$ is even,} \\
    \dst z^2-\frac{p^n-p^2}{p^n-1}, & \hbox{if $n$ is odd,}
  \end{array}
\right.
$$
with $\lim_{n\to\infty}\lambda_n=1$. Now, $R_2^*(f)$ may be worked
out by means of formula \eqref{radio} according to Lemma
\ref{telescopica}. Keeping in mind the notation adopted there and
using the expression of $Q_{n,\mathbf{m}}$ calculated before, it
turns out that
$$
|A_{n,2}|=\lambda_n\frac{p(p^2-1)}{p^{n+1}-1},\qquad n\;\,
\mbox{even}.
$$
Then, $ \lim_{n=2{\mathbb{Z}}_+} |A_{n,2}|^{1/n}=1/p, $ which implies $$
p \leq R_2^*(f)=\frac{1}{\dst\limsup_{n\to\infty} |A_{n,2}|^{1/n}}\le p. $$

Thus, we have proved that $R^*_2(f)=p$ may take any value between
$R_1(f)= 1$ and $R_2(f) = \infty$, both ends included.

\subsection{Comparison between the Graves-Morris/Saff Theorem and Theorem \ref{saff2}}
\label{ejemplo3}

First, let us see that there are very simple systems $\mathbf{f}$
that are not polewise independent in $D_{|\mathbf{m}|}(\mathbf{f})$
 and still they have exactly
$|\mathbf{m}|$ poles in $\mathbf{D}_{\mathbf{m}}^*(\mathbf{f})$. Set
\[
f_1(z)=\frac{1}{1-z}+\frac{1}{2-z},\qquad f_2(z)=\frac{1}{3-z},
\]
and fix the multi-index $\mathbf{m}=(1,1)$. Put
$\mathbf{f}=(f_1,f_2)$. It is clear that $R_2(\mathbf{f})=3$ and, as
$0f_1+f_2$ is analytic in $D_2(\mathbf{f})$, the system $\mathbf{f}$
is not polewise independent in $D_{2}(\mathbf{f})$. Also, as
$R_1(f_2)=\infty$, we have $R_2^*(f_2)=\infty$ and one of the poles
of $Q_{n,\mathbf{m}}$ is attracted by the point $z=3$ on account of
Corollary \ref{convergence2}. On the other hand, $R_1(f_1)=2$, so
$R_2^*(f_1)\ge 2$ but $R_2^*(f_1)$ cannot be greater than $2$ since
in that case two other poles of  $Q_{n,\mathbf{m}}$ would be
attracted by the points $z=1$ and $z=2$, which is absurd. Then
$R_2^*(f_1)= 2$ and the system $\mathbf{f}$ has exactly two poles,
$z=1$ and $z=3$, in $\mathbf{D}_{\mathbf{m}}^*(\mathbf{f})$.
%We will use below that Theorem \ref{saff2} implies that
%\begin{equation}\label{vel1}
%\limsup_{n \to \infty} \|\mathcal{Q}_{\mathbf{m}}(\mathbf{f}) -
%Q_{n,{\bf m}}\|^{1/n}\le 1/2.
%\end{equation}
This example also shows that the inequality appearing in part c) of
Lemma \ref{lema2} may be strict.

Now, fix again $\mathbf{m}=(1,1)$ and take $\mathbf{g}=(g_1,g_2)$,
where
$$
g_1(z)=\frac{1}{1-z}+\log (3-z),\qquad g_2(z)=\frac{1}{2-z}+\log
(10-z).
$$
Obviously, $R_1(g_1)=R_2^*(g_1)=R_2(g_1)=3$ and
$R_1(g_2)=R_2^*(g_2)=R_2(g_2)=10$. The system $\mathbf{g}$ is
polewise independent in $D_2(\mathbf{g})$ with $R_2(\mathbf{g})=3$.
The Morris-Graves/Saff Theorem gives
$$
\limsup_{n \to \infty} \|g_2 - R_{n,{\bf m},2}\|_{K}^{1/n} \leq
\frac{\|z\|_ {K}}{3},
$$
for any compact subset $K$ of $\{z\,:\, |z|<3\}$ and
$$
\limsup_{n \to \infty} \|\mathcal{Q}_{\mathbf{m}}(\mathbf{g}) -
Q_{n,{\bf m}}\|^{1/n}\le 2/3,
$$
where $\mathcal{Q}_{\mathbf{m}}(\mathbf{g})(z)=(1-z)(1-z/2)$. On the
other hand, Theorem \ref{saff2} gives
$$
\limsup_{n \to \infty} \|g_2 - R_{n,{\bf m},2}\|_{K}^{1/n} \leq
\frac{\|z\|_ {K}}{10},
$$
for any compact subset $K$ of $\{z\,:\, |z|<10\}$ and
$$
\limsup_{n \to \infty} \|\mathcal{Q}_{\mathbf{m}}(\mathbf{g}) -
Q_{n,{\bf m}}\|^{1/n}\le \max\{1/3,1/5\}=1/3.
$$

\subsection{On the rate of convergence of $\{Q_{n,\mathbf{m}}\}$}
\label{ejemplo4}

Let us show that the rate of convergence of the sequence of polynomials $Q_{n,\mathbf{m}}$
given by the inequality \eqref{inequality2} is not exact in general. Fix
$\mathbf{m}=(1,1)$ and consider the system $\mathbf{h}=(h_1,h_2)$,
where
$$
h_1(z)=\frac{1}{1-z}+\frac{1}{2-z}+\log(3-z),\quad
h_2(z)=\frac{1}{1-z}+\log(3-z)+\log(4-z).
$$
Obviously $R_2(\mathbf{h})=3$ and the system $\mathbf{h}$ is
polewise independent in $D_{2}(\mathbf{h})$. As
$R_1(h_2)=R_2(h_2)=3$, we have $R_2^*(h_2)=3$. On the other hand, we have that
$R_2(h_1)=3$, from which it follows that $R_2^*(h_1)\le 3$. Using
part d) of Lemma \ref{lema2}, we obtain $R_2^*(h_1)=3$. Therefore,
Theorem \ref{saff2} gives
$$
\limsup_{n \to \infty} \|\mathcal{Q}_{\mathbf{m}}(\mathbf{h}) -
Q_{n,{\bf m}}\|^{1/n}\le \max\{1/3,2/3\}=2/3,
$$
where $\mathcal{Q}_{\mathbf{m}}(\mathbf{h})(z)=(1-z)(1-z/2)$.

Consider now the system $\mathbf{\hat{h}}=(\hat{h}_1,\hat{h}_2)$,
where $\hat{h}_1=h_1-h_2$ and $\hat{h}_2=h_2$. We have
$R_1(\hat{h}_1)=4=R_2(\hat{h}_1)$, hence $R_2^*(\hat{h}_1)=4$. As
before, $R_2^*(\hat{h}_2)=3$. Obviously, the $(n,\mathbf{m})$  Hermite-Pad\'e
approximants of the systems $\mathbf{h}$ and $\mathbf{\hat{h}}$ have
the same common denominator $Q_{n,{\bf m}}$. Using again
Theorem \ref{saff2} for the new auxiliary system, we obtain a better estimate
$$
\limsup_{n \to \infty} \|\mathcal{Q}_{\mathbf{m}}(\mathbf{h}) -
Q_{n,{\bf m}}\|^{1/n}\le \max\{1/3,2/4\}=1/2.
$$

%%%%%%%%%%%%%%%%%%%%%%%%%%%%%%%%%%%%%%%%%%%%%%%%%%%%%%%%%%%%%%%%%%%%%%%%%%%%%%%%%%%%%%%%%%%%%%%%%%%%%%%%%%%%%%%%%%
%%%%%%%%%%%%%%%%%%%%%%%%%%%%%%%%%%%%%%%%%%%%%%%%%%%%%%%%%%%%%%%%%%%%%%%%%%%%%%%%%%%%%%%%%%%%%%%%%%%%%%%%%%%%%%%%%%
%%%%%%%%%%%%%%%%%%%%%%%%%%%%%%%%      REFERENCES    %%%%%%%%%%%%%%%%%%%%%%%%%%%%%%%%%%%%%%%%%%%%%
%%%%%%%%%%%%%%%%%%%%%%%%%%%%%%%%%%%%%%%%%%%%%%%%%%%%%%%%%%%%%%%%%%%%%%%%%%%%%%%%%%%%%%%%%%%%%%%%%%%%%%%%%%%%%%%%%%%
%%%%%%%%%%%%%%%%%%%%%%%%%%%%%%%%%%%%%%%%%%%%%%%%%%%%%%%%%%%%%%%%%%%%%%%%%%%%%%%%%%%%%%%%%%%%%%%%%%%%%%%%%%%%%%%%%%

%% References
%%
%% Following citation commands can be used in the body text:
%% Usage of \cite is as follows:
%%   \cite{key}         ==>>  [#]
%%   \cite[chap. 2]{key} ==>> [#, chap. 2]
%%


\begin{thebibliography}{50}

\bibitem{bre} C. Brezinski,  Pad\'e-Type Approximation and
General Orthogonal Polynomials, Birkh\"{a}user, Basel, 1980.

\bibitem{gon} A.A. Gonchar,  On the convergence of generalized
Pad\'e approximants of meromorphic functions, Sb. Math. 27 (1975)
503--514.

\bibitem{gon2} A.A. Gonchar,  Poles of rows of the Pad\'e table and
meromorphic continuation of functions, Sb. Math. 43 (1982) 527--546.

\bibitem{gon3} A.A. Gonchar, L.D. Grigorjan, On estimates of
the norm of the holomorphic component of a meromorphic function, Sb.
Math. 28 (1976) 571--575.

\bibitem{Had} J. Hadamard, Essai sur l'\'etude des fonctions
donn\'ees par leur d\'eveloppement de Taylor, J. Math. Pures Appl.
 8 (1892) 101--186.

\bibitem{marden} M. Marden,
Geometry of Polynomials,  Amer. Math. Soc., Providence, Rhode
Island,  1949.

\bibitem{Mon} R. de Montessus de Ballore, Sur les fractions continues
alg\'ebriques, Bull. Soc. Math. France 30 (1902) 28--36.

\bibitem{GS1} P.R. Graves-Morris, E.B. Saff, A de Montessus
theorem for vector-valued rational interpolants, Lecture Notes in
Math. 1105, Springer, Berlin, 1984, pp. 227--242.

\bibitem{GS2} P.R. Graves-Morris, E.B. Saff, Row convergence
theorems for generalized inverse vector-valued Pad\'e approximants,
J. Comp. Appl. Math. 23 (1988) 63--85.

\bibitem{GS3} P.R. Graves-Morris, E.B. Saff, An extension of a
row convergence theorem for vector Pad\'e approximants, J. Comp.
Appl. Math. 34 (1991) 315--324.

\bibitem{small} T. Sheil-Small,
Complex Polynomials,  Cambridge University Press, Cambridge, 2002.


\end{thebibliography}
\end{document}